\documentclass[12pt]{amsart}
\usepackage{amsmath, amssymb, amsthm, amsfonts}
\usepackage{paralist, xcolor, hyperref}

\usepackage[margin=1in,letterpaper,portrait]{geometry}

\usepackage{enumitem}
\usepackage{cite}
\usepackage{array}

\usepackage{tikz}
\usetikzlibrary{patterns}

\newtheorem{conjecture}{Conjecture}[section]
\newtheorem{theorem}[conjecture]{Theorem}
\newtheorem{corollary}[conjecture]{Corollary}
\newtheorem{lemma}[conjecture]{Lemma}
\newtheorem{proposition}[conjecture]{Proposition}

\theoremstyle{definition}

\newtheorem{example}[conjecture]{Example}

\newtheorem{definition}[conjecture]{Definition}
\newtheorem*{question*}{Question}

\newcommand{\symm}{\mathfrak{S}}
\newcommand{\s}{\sigma}
\newcommand{\support}{\textup{\textsf{supp}}}

\newcommand{\rw}[1]{\mbox{$[#1]$}}

\title{Prism permutations in the Bruhat order}

\author{Bridget Eileen Tenner}
\address{Department of Mathematical Sciences, DePaul University, Chicago, IL, USA}
\email{bridget@math.depaul.edu}
\thanks{Research partially supported by NSF Grant DMS-2054436.}

\keywords{}%

\makeatletter
\@namedef{subjclassname@2020}{%
  \textup{2020} Mathematics Subject Classification}
\makeatother
\subjclass[2020]{Primary: 06A07; 
  Secondary: 05A05, 
  20F55 
  }

\begin{document}

\begin{abstract}
The boolean elements of a Coxeter group have been characterized and shown to possess many interesting properties and applications. Here we introduce ``prism permutations,'' a generalization of those elements, characterizing 
the prism permutations equivalently in terms of their reduced words and in terms of pattern containment. As part of this work, we introduce the notion of ``calibration'' to permutation patterns.
\end{abstract}

\maketitle

Boolean principal order ideals in the Bruhat order have been studied extensively \cite{gao hanni, hultman vorwerk, mazorchuk tenner, ragnarsson tenner homotopy, ragnarsson tenner homology, tenner patt-bru}. The elements corresponding to those principal order ideals have been shown to have important combinatorial, topological, and representation theoretic properties. Most of that previous work has focused on the symmetric group (the finite Coxeter group of type $A$), although the initial characterization was done in a broader context \cite{tenner patt-bru} and this has been studied from another perspective more recently \cite{gao hanni}. People have also been interested in boolean ideals and intervals more generally, as in \cite{elder harris kretschmann martinez mori, tenner intervals, bool-intersec-interlac}. 

One of the most advantageous results about the so-called ``boolean'' permutations is that they are characterized by pattern avoidance. In particular, boolean elements in $\symm_n$ are exactly those permutations that avoid both $321$ and $3412$. In the present work, we generalize those previous efforts by studying what we call \emph{prisms}: permutations whose principal order ideals can be written as the direct product of a nontrivial boolean algebra and another poset. As in the original setting, this class of objects can be described both in terms of reduced words and in terms of patterns, although classical pattern containment is no longer a sufficient system for the purpose.

Our first main result (Theorem~\ref{thm:reduced word characterization}) is that a permutation is a prism if and only if it has a reduced word in which there is some $i$ that appears exactly once, not between two copies of $i+1$, nor between two copies of $i-1$. In contrast, a permutation is purely boolean if and only if \emph{every} letter in its reduced word appears exactly once (and thus there is no possibility of appearing between two copies of some other letter). Our second main result (Theorem~\ref{thm:pattern characterization}) is that a permutation is a prism if and only if it contains one of eight ``calibrated'' patterns. It is interesting to note that purely boolean elements are characterized by pattern \emph{avoidance}, while prisms are characterized by pattern \emph{containment}. While this might initially seem to be counterintuitive, it is a consequence of prisms requiring a nontrivial boolean factor. It is also worth noting that the eight calibrated patterns do prohibit certain copies of the $321$- and $3412$-patterns that are universally prohibited in purely boolean elements.

\section{Introduction}\label{sec:intro}

Permutations in $\symm_n$ can be written as products of simple reflections $\{\s_i : i \in [1,n-1]\}$, where $\s_i$ is the permutation transposing $i$ and $i+1$, and fixing all other values. For a permutation $w$, the minimum number of simple reflections needed for such a product to equal $w$ is the \emph{length} of $w$, denoted $\ell(w)$. 

\begin{definition}\label{defn:reduced words}
If $w = \s_{i_1}\cdots\s_{i_{\ell(w)}}$, then $ \s_{i_1}\cdots\s_{i_{\ell(w)}}$ is a \emph{reduced decomposition} of $w$. This corresponds to the, equivalently informative, \emph{reduced word} $\rw{i_1 \cdots i_{\ell(w)}}$ of $w$. The collection of all reduced words of $w$ is $R(w)$. The \emph{support} of $w$, denoted $\support(w)$, is the set $\{i_1, \ldots, i_{\ell(w)}\}$ of indices appearing in any reduced word of $w$.
\end{definition}

Elements of $R(w)$ are related by \emph{commutations} ($\rw{ij} = \rw{ji}$ when $|i-j| > 1$) and \emph{braids} ($\rw{i(i+1)i} = \rw{(i+1)i(i+1)}$). Thus $\support(w)$ is well defined, and does not depend on the particular reduced word being considered. We think of permutations as maps, and interpret their products as compositions.

Permutations in $\symm_n$ can also be written in one-line notation, as words of the form
$$w(1) \cdots w(n),$$
or as graphs $G(w) = \{(x,w(x)) : 1 \le x \le n\}$. The notation $\rw{s}$ for reduced words is meant to distinguish, for example, the reduced word $\rw{123} = \s_1 \s_2 \s_3 = 2341 \in \symm_4$ from the permutation $123 \in \symm_3$.
Both one-line notation and graphs are well-suited to the study of permutation patterns.

\begin{definition}\label{defn:pattern}
Fix $p \in \symm_k$ and $w \in \symm_n$. If there exist indices $1 \le i_1 < \cdots < i_k \le n$ such that the subword $w(i_1)\cdots w(i_k)$ is in the same relative order as $p(1)\cdots p(k)$, then $w$ \emph{contains} a $p$-pattern and we have found an \emph{occurrence} of $p$. Otherwise, $w$ \emph{avoids} $p$ or is \emph{$p$-avoiding}. When $w$ contains $p$ as described, then that occurrence's \emph{positions} are $\{i_1, \ldots, i_k\}$ and its \emph{values} are $\{w(i_1), \ldots, w(i_k)\}$.
\end{definition}

We will find it useful to reference such specific features of pattern occurrences.

\begin{example}
The permutation $453261$ is $312$-avoiding and contains several $231$-patterns. One of these has positions $\{3, 5, 6\}$ and values $\{3, 6, 1\}$. 
\end{example}

While one-line notation and graphs are good frameworks for questions about permutation patterns, reduced words define an important poset on $\symm_n$.

\begin{definition}\label{defn:bruhat order}
The \emph{Bruhat order} gives a partial ordering to $\symm_n$, defined so that
$$v \preceq w$$
if and only if there exist $\rw{s} \in R(v)$ and $\rw{t} \in R(w)$ such that $s$ is a subword of $t$.
\end{definition}

The Bruhat order is an important structure with many interesting features -- some well understood and others less so. The curious reader is encouraged to begin exploring this topic using \cite[Chapter 2]{bjorner brenti}.

Our interest here relates to the principal order ideals of permutations in this poset; for $w \in \symm_n$, we write
$$B(w) := \{v : v \preceq w \text{ in the Bruhat order}\}.$$
The Bruhat order has a complicated structure in general, although some of its elements have particularly ``nice'' principal order ideals. In previous work, we characterized those elements.

\begin{theorem}[\!\!{\cite[Theorem 4.3]{tenner patt-bru}}]\label{thm:boolean poi}
For a permutation $w \in \symm_n$, the following statements are equivalent:
\begin{itemize}
\item $B(w)$ is isomorphic to a boolean algebra,
\item every $\rw{s} \in R(w)$ has the property that $s$ has no repeated letters,
\item there exists $\rw{s} \in R(w)$ for which $s$ has no repeated letters, and
\item $w$ is $321$- and $3412$-avoiding.
\end{itemize}
\end{theorem}

Theorem~\ref{thm:boolean poi} led to the notion of \emph{boolean permutations} (more generally, \emph{boolean elements} in any Coxeter group), which are exactly those permutations satisfying the properties listed in Theorem~\ref{thm:boolean poi}. In this work, we generalize that result to consider a broader class of permutations.

\begin{definition}\label{defn:new class}
A permutation $w$ is a \emph{prism} if $B(w) \cong B \times X$ for some nontrivial boolean algebra $B$. Due to the associativity of direct products (see, for example, \cite[\S3.2]{ec1}), it suffices to consider $B = 
\vcenter{\hbox{\begin{tikzpicture}[scale=.5]
\foreach \y in {0,1} {\fill (0,\y) circle (3pt);}
\draw (0,0) -- (0,1);
\end{tikzpicture}}}$.
\end{definition}

Note that, like boolean elements, prisms can be defined for any Coxeter group. Additionally, any prism that can be written with $X = \vcenter{\hbox{\begin{tikzpicture}[scale=.5] \fill (0,0) circle (3pt); \end{tikzpicture}}}$ is also ``purely'' boolean. The only purely boolean element that is not a prism is the identity permutation $e$.

\begin{example}
Table~\ref{tab:s4 categories} classifies each element of $\symm_4$ as boolean, a prism, or neither, and Figure~\ref{fig:picture of s4} marks them in the Hasse diagram of $\symm_4$. Foreshadowing the main results of this work, we note that
$$B(4132) \cong B(2431) \cong \vcenter{\hbox{\begin{tikzpicture}[scale=.5]
\foreach \y in {0,1} {\fill (0,\y) circle (3pt);}
\draw (0,0) -- (0,1);
\end{tikzpicture}}}
\times B(1432)
\hspace{.25in} \text{and} \hspace{.25in}
B(4213) \cong B(3241) \cong \vcenter{\hbox{\begin{tikzpicture}[scale=.5]
\foreach \y in {0,1} {\fill (0,\y) circle (3pt);}
\draw (0,0) -- (0,1);
\end{tikzpicture}}}
\times B(3214).$$

\begin{table}[htbp]
{\renewcommand{\arraystretch}{1.5}
$\begin{array}{ll|l|l}
\vspace{-.1in} \ \hspace{1in} \ & & \text{Prism} 
& \text{Neither boolean} \\
\multicolumn{2}{c|}{\text{Boolean}} & \text{but not boolean} & \text{nor prism}\\
\hline
\hline
1234 = \rw{\emptyset} & 2314 = \rw{12} & 2431 = \rw{1232} & 1432 = \rw{232}\\
1243 = \rw{3} & 2341 = \rw{123} & 3241 = \rw{1213} & 3214 = \rw{121}\\
1324 = \rw{2} & 2413 = \rw{312} & 4132 = \rw{2321} & 3412 = \rw{2132}\\
1342 = \rw{23} & 3124 = \rw{21} & 4213 = \rw{3121} & 3421 = \rw{21232}\\
1423 = \rw{32} & 3142 = \rw{213} & & 4231 = \rw{12321}\\
2134 = \rw{1} & 4123 = \rw{321} & & 4312 = \rw{23212}\\
2143 = \rw{13} & & & 4321 = \rw{123121}
\end{array}$
}
\vspace{.1in}
\caption{The elements of $\symm_4$, identified as boolean, prisms,
and neither. In each case, a reduced word of the permutation (sometimes one of many) is also given.}\label{tab:s4 categories}
\end{table}

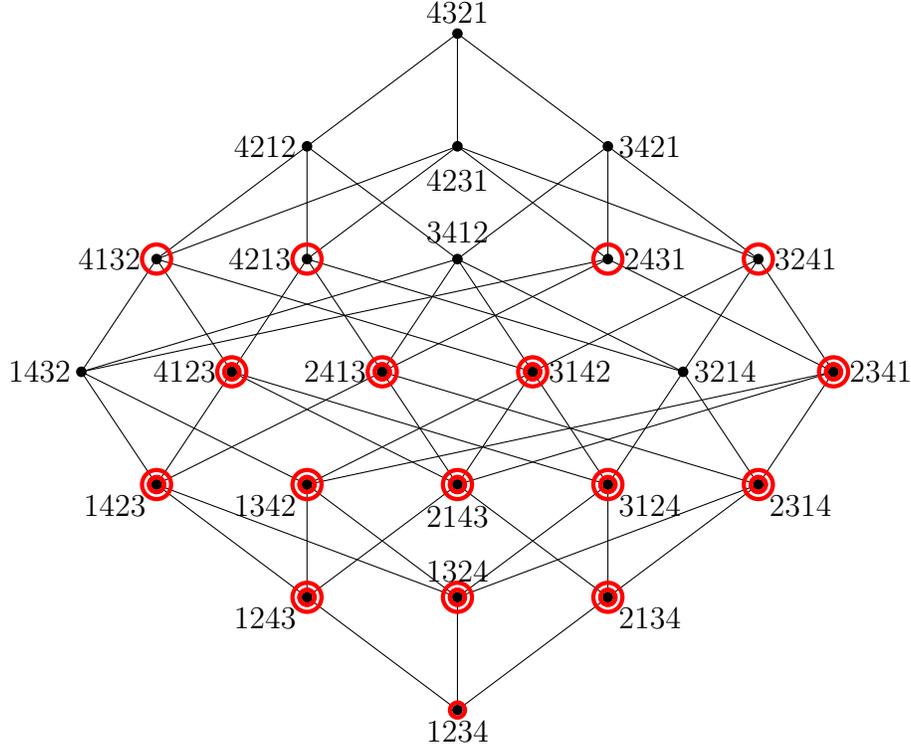
\begin{figure}[htbp]
\begin{tikzpicture}[scale=1]
\foreach \y in {0,9} {\fill[black] (0,\y) circle (2pt);}
\foreach \x in {-2,0,2} {\foreach \y in {1.5,7.5} {\fill[black] (\x,\y) circle (2pt);};}
\foreach \x in {-4,-2,0,2,4} {\foreach \y in {3,6} {\fill[black] (\x,\y) circle (2pt);};}
\foreach \x in {-5,-3,-1,1,3,5} {\fill[black] (\x,4.5) circle (2pt);}
\draw (0,0) coordinate (1234); \draw (0,9) coordinate (4321);
\draw (-2,1.5) coordinate (1243); \draw (0,1.5) coordinate (1324); \draw (2,1.5) coordinate (2134);
\draw (-2,7.5) coordinate (4312); \draw (0,7.5) coordinate (4231); \draw (2,7.5) coordinate (3421);
\draw (-4,3) coordinate (1423); \draw (-2,3) coordinate (1342); \draw (0,3) coordinate (2143); \draw (2,3) coordinate (3124); \draw (4,3) coordinate (2314);
\draw (-4,6) coordinate (4132); \draw (-2,6) coordinate (4213); \draw (0,6) coordinate (3412); \draw (2,6) coordinate (2431); \draw (4,6) coordinate (3241);
\draw (-5,4.5) coordinate (1432); \draw (-3,4.5) coordinate (4123); \draw (-1,4.5) coordinate (2413); \draw (1,4.5) coordinate (3142); \draw (3,4.5) coordinate (3214); \draw (5,4.5) coordinate (2341);
\draw (1234) -- (1243) -- (1423) -- (1432) -- (4132) -- (4312) -- (4321) -- (3421) -- (3241) -- (2341) -- (2314) -- (2134) -- (1234) -- (1324) -- (1423) -- (4123) -- (4132) -- (4231) -- (4321);
\draw (1243) -- (1342) -- (1432) -- (3412) -- (4312);
\draw (1243) -- (2143) -- (4123) -- (4213) -- (4312);
\draw (1324) -- (1342) -- (3142) -- (4132);
\draw (1324) -- (3124) -- (4123);
\draw (1324) -- (2314) -- (2413) -- (4213) -- (4231);
\draw (2134) -- (2143) -- (2413) -- (3412) -- (3421);
\draw (2134) -- (3124) -- (3142) -- (3412);
\draw (1423) -- (2413) -- (2431) -- (4231);
\draw (1342) -- (2341) -- (2431) -- (3421);
\draw (2143) -- (3142) -- (3241) -- (4231);
\draw (2143) -- (2341);
\draw (3124) -- (3214) -- (4213);
\draw (2314) -- (3214) -- (3412);
\draw (1432) -- (2431);
\draw (3214) -- (3241);
\foreach \x in {1234,2134,1324,1243,1423,1342,2143,3124,2314,4123,2413,3142,2341} {\draw[ultra thick,red] (\x) circle (2.85pt);}
\foreach \x in {2134,1324,1243,1423,1342,2143,3124,2314,4123,2413,3142,2341} {\draw[ultra thick,red] (\x) circle (2.85pt); \draw[ultra thick, red] (\x) circle (5.5pt);}
\foreach \x in {2431,3241,4132,4213} {
\draw[ultra thick,red] (\x) circle (5.5pt);}
\draw (1234) node[below] {$1234$};
\draw (4321) node[above] {$4321$};
\draw (2134) node[below right] {$2134$}; \draw (1324) node[above, yshift=2pt] {$1324$}; \draw (1243) node[below left] {$1243$};
\draw (1342) node[below left] {$1342$}; \draw (1423) node[below left] {$1423$}; \draw (2143) node[below, yshift=-4pt] {$2143$}; \draw (2314) node[below right] {$2314$}; \draw (3124) node[below right] {$3124$};
\draw (1432) node[left] {$1432$}; \draw (3214) node[right] {$3214$}; \draw (4123) node[left, xshift=-2pt] {$4123$}; \draw (2341) node[right, xshift=2pt] {$2341$}; \draw (2413) node[left, xshift=-2pt] {$2413$}; \draw (3142) node[right, xshift=2pt] {$3142$};
\draw (4132) node[left, xshift=-2pt] {$4132$}; \draw (4213) node[left, xshift=-2pt] {$4213$}; \draw (3241) node[right, xshift=2pt] {$3241$}; \draw (2431) node[right, xshift=2pt] {$2431$}; \draw (3412) node[above, yshift=2pt] {$3412$};
\draw (4231) node[below, yshift=-5pt] {$4231$}; \draw (4312) node[left] {$4212$}; \draw (3421) node[right] {$3421$};
\end{tikzpicture}
\caption{The Bruhat order of $\symm_4$. The boolean elements are marked with small red circles. The prisms are marked with large red circles.}\label{fig:picture of s4}
\end{figure}
\end{example}

In this work, we characterize prisms in two ways: first using the language of reduced words in Section~\ref{sec:rw characterization}, and then using the language of permutation patterns in Section~\ref{sec:pattern characterization}. The latter characterization will introduce \emph{calibrated} permutation patterns (Definition~\ref{defn:calibrated}). The two equivalent characterizations of prisms, presented in Theorems~\ref{thm:reduced word characterization} (equivalently, Corollary~\ref{cor:new class word forms}) and Theorem~\ref{thm:pattern characterization}, echo and extend previous work translating between pattern properties and properties of reduced words \cite{tenner patt-bru, tenner rep-patt, tenner rep-range, tenner rwm}. We conclude with Section~\ref{sec:further directions}, proposing several directions for further study.

\section{Characterizing prisms via reduced words}\label{sec:rw characterization}

Our first analysis of prisms is from the perspective of reduced words. We begin by establishing that prisms are not just isomorphic to the product of a boolean algebra and a generic poset $X$, but that $X$ is actually the principal order ideal of a particular smaller permutation.

\begin{lemma}\label{lem:new class means one factor atom and the other in the poset}
A permutation $w$ is a prism if and only if there exists $v \prec w$ with $\support(v) \subsetneq \support(w)$ and
$$B(w) \cong \vcenter{\hbox{\begin{tikzpicture}[scale=.5]
\foreach \y in {0,1} {\fill (0,\y) circle (3pt);}
\draw (0,0) -- (0,1);
\end{tikzpicture}}} \times B(v).$$
Moreover, $\support(w) \setminus \support(v) = \{i\}$ for some $i$, and $v$ is obtained by deleting a copy of $i$ from a reduced word for $w$. To emphasize the role of this $i$, we will write $B(w) \cong B(\s_i) \times B(v)$.
\end{lemma}

\begin{proof}
If $B(w)$ has the given form, then certainly $w$ is a prism.

Now suppose that $w$ is a prism. Then there is an isomorphism
$$\alpha : B(w) \longrightarrow 
\left(\vcenter{\hbox{\begin{tikzpicture}[scale=.75]
\foreach \y in {0,1} {\fill (0,\y) circle (2pt);}
\draw (0,0) -- (0,1);
\draw (0,0) node[left] {$0$};
\draw (0,1) node[left] {$1$};
\end{tikzpicture}}}
\times X'\right).$$
The poset $B(w)$ is bounded, so $\vcenter{\hbox{\begin{tikzpicture}[scale=.5]
\foreach \y in {0,1} {\fill (0,\y) circle (3pt);}
\draw (0,.2) node[left] {$\scriptstyle{0}$};
\draw (0,.8) node[left] {$\scriptstyle{1}$};
\draw (0,0) -- (0,1);
\end{tikzpicture}}} \times X'$ is also bounded. Let its maximum and minimum be $(1,m)$ and $(0,e)$, respectively. Thus $m$ is the maximum of $X'$ and $e$ is the minimum. 
Define $i \in \support(w)$ and $v \prec w$ so that
$$\s_i := \alpha^{-1} \left((1,e)\right) \hspace{.25in} \text{and} \hspace{.25in} v := \alpha^{-1} \left((0, m)\right).$$
The map $\alpha$ is an isomorphism, so the principal order ideal $B(v) \subset B(w)$ is isomorphic to $\{0\} \times X' \cong X'$. Hence $B(w) \cong B(\s_i) \times B(v)$.

The element $\alpha(\s_i) = (1,e)$ is not less than $\alpha(v) = (0,m)$ in $B(\s_i) \times B(v)$. Thus $i \not\in\support(v)$. 

The element $(0,m)$ is a coatom in $B(\s_i) \times B(v)$, so $v$ is a coatom in $B(w)$. In other words, there is a reduced word for $v$ that is obtained by deleting a single letter from a reduced word for $w$. Since $i \in \support(w) \setminus \support(v)$ (in fact, $\{i\} = \support(w) \setminus \support(v)$), that letter is $i$.
\end{proof}

This has an immediate implication for the reduced words of a prism. 

\begin{corollary}\label{cor:s_i appears once}
Suppose that $B(w) \cong B(\s_i) \times B(v)$ as defined in Lemma~\ref{lem:new class means one factor atom and the other in the poset}. Then $i$ appears exactly once in all reduced words of $w$.
\end{corollary}

\begin{proof}
Suppose, for the purpose of obtaining a contradiction, that there is some $\rw{s} \in R(w)$ with at least two copies of $i$. These must be separated by at least one $i\pm1$, so $B(w)$ contains a principal order ideal isomorphic to $B(321)$, where one of its atoms corresponds to $(\s_i,e) \in B(\s_i) \times B(v)$, and this is impossible.
\end{proof}

In order to appreciate the implications of Corollary~\ref{cor:s_i appears once}, we use two previous results about the repetition of letters in reduced words.

\begin{lemma}[\!\!{\cite[Lemma~2.8]{tenner rep-patt}}]\label{lem:support rules}
Fix a permutation $w \in \symm_n$. The following are equivalent:
\begin{itemize}
\item $i \in \support(w)$,
\item $\{w(1),\ldots, w(i)\} \neq \{1,\ldots, i\}$,
\item $\{w(i+1), \ldots, w(n)\} \neq \{i+1,\ldots, n\}$, 
\item $w$ contains a $21$-pattern in positions $\{x_1,x_2\}$ such that $x_1 \le i < x_2$ and with values $\{y_1,y_2\}$ such that $y_1 \le i < y_2$.
\end{itemize}
\end{lemma}

\begin{lemma}[\!\!{\cite[Theorem~3.3]{tenner rep-range}}]\label{lem:support range}
Fix a permutation $w \in \symm_n$ and $i \in \support(w)$. Then $i$ appears exactly once in all elements of $R(w)$ if and only if:
\begin{itemize}
\item $w$ has no $321$-pattern with positions $x_1 < x_2 < x_3$ satisfying $x_1 \le i < x_3$ and values $y_1 < y_2 < y_3$ satisfying $y_1 \le i < y_3$, and
\item $w$ has no $3412$-pattern with positions $x_1 < x_2 < x_3 < x_4$ satisfying $x_2 \le i < x_3$ and values $y_1 < y_2 < y_3 < y_4$ satisfying $y_2 \le i < y_3$.
\end{itemize}
\end{lemma}

We are now ready to understand one consequence of Corollary~\ref{cor:s_i appears once}. In fact, it is a special case of the following, more general, result.

\begin{corollary}\label{cor:appearing once}
Suppose that $w$ is a permutation in which $i \in \support(w)$ appears exactly once in all elements of $R(w)$. Then deleting $i$ from any element of $R(w)$ yields a word that is still reduced, and all such deletions are reduced words for the same permutation $v$.
\end{corollary}

\begin{proof}
Fix $\rw{s} \rw{i} \rw{t} \in R(w)$, so $i \not\in s$ and $i \not\in t$. Lemma~\ref{lem:support rules} means that $\rw{s} \in R(u)$ for
$$u = \framebox{ \text{permutation of } $[1,i]$ } \ \framebox{ \text{permutation of } $[i+1,n]$ } .$$
Set $a := u(i) \le i$ and $b := u(i+1) \ge i+1$. Thus $\rw{s} \rw{i}$ is a reduced word for
$$u \s_i = \framebox{ \text{permutation of } $[1,i] \setminus \{a\}$ } \ b \ a \ \framebox{ \text{permutation of } $[i+1,n]\setminus \{b\}$ } ,$$
and $\rw{s} \rw{i} \rw{t} \in R(w)$ describes
$$w = \framebox{ \text{permutation of } $[1,i] \setminus \{a\} \cup \{b\}$ } \ \framebox{ \text{permutation of } $[i+1,n] \setminus \{b\} \cup \{a\}$ } .$$
Let $v$ be the permutation obtained from $w$ by swapping the values $a$ and $b$ in the one-line notation of $v$. Note that $b = \max\{w(1),\ldots, w(i)\}$ and $a = \min\{w(i+1),\ldots,w(n)\}$, so $v$ depends only on $w$ and $i$, and not on $s$ or $t$. The length of a permutation is equal to the number of its inversions, and it follows from the first item of Lemma~\ref{lem:support range} that $v$ has exactly one fewer inversion than $w$ has. Thus $\ell(v) = \ell(w) - 1$, and so $\rw{s}\rw{t} \in R(v)$. 
\end{proof}

Corollary~\ref{cor:appearing once} does not require the permutation $w$ to be a prism, but it does ensure that the permutation $v$ described in Lemma~\ref{lem:new class means one factor atom and the other in the poset} is defined entirely by $w$ and $i$, and not by a particular choice of reduced word for $w$.

We illustrate the results so far with an example from Table~\ref{tab:s4 categories}.

\begin{example}
For $w = 2431$, we can use $i = 1$ and $v = 1432$ to see that
$$B(2431) \cong B(\s_1) \times B(1432).$$
Note that $R(2431) = \{\rw{1232}, \rw{1323}, \rw{3123}\}$, the permutation $1432$ can be obtained by deleting $1$ from any of $w$'s reduced words: $R(1432) = \{\rw{232}, \rw{323}\}$, and $1 \not\in \support(1432)$. The poset $B(2431)$ is drawn in Figure~\ref{fig:2431 example}, and colored to highlight its decomposition as $B(\s_1) \times B(1432)$.
\begin{figure}[htbp]
\begin{tikzpicture}
\draw (0,0) coordinate (e);
\draw (-1,1) coordinate (1);
\draw (-0,1) coordinate (2);
\draw (1,1) coordinate (3);
\draw (-1.5,2) coordinate (12);
\draw (-.5,2) coordinate (13);
\draw (.5,2) coordinate (23);
\draw (1.5,2) coordinate (32);
\draw (-1,3) coordinate (123);
\draw (0,3) coordinate (132);
\draw (1,3) coordinate (232);
\draw (0,4) coordinate (1232);
\foreach \x in {e,1,2,3,12,13,23,32,132,123,232,1232} {\fill (\x) circle (2pt);}
\draw[thick] (3) -- (e) -- (2) -- (23) -- (3) -- (32) -- (2);
\draw[thick] (23) -- (232) -- (32);
\draw[dotted, thick] (13) -- (1) -- (12) -- (123) -- (13) -- (132) -- (12);
\draw[dotted, thick] (123) -- (1232) -- (132);
\foreach \x in {2,3,23,32,232} {\draw[red, thick] (\x) -- (1\x);}
\draw[red, thick] (e) -- (1);
\end{tikzpicture}
\caption{The principal order ideal $B(2431) \cong B(\s_1) \times B(1432)$. Covering relations corresponding to $\{e\} \times B(1432)$ are solid black lines, those corresponding to $\{\s_1\} \times B(1432)$ are dotted black lines, and those corresponding to $(e, x) \lessdot (\s_1,x)$ are draw in red.}\label{fig:2431 example}
\end{figure}
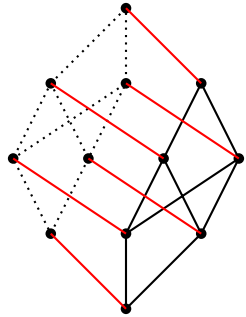
\end{example}

When $w$ is a prism, there is yet more to say about the appearance of this $i$ in elements of $R(w)$.

\begin{definition}\label{defn:confined}
Fix a permutation $w$, a reduced word $\rw{s} \in R(w)$, and a letter $i \in \support(w)$ that appears exactly once in $\rw{s}$. If that appearance is not between two copies of $i+1$ and not between two copies of $i-1$, then $i$ is \emph{unconfined} in $\rw{s}$.
\end{definition}

Being unconfined is independent of the choice of reduced word $\rw{s}$.

\begin{lemma}\label{lem:unconfined everywhere}
Fix a permutation $w$ and a letter $i \in \support(w)$. If $i$ is unconfined in $\rw{s} \in R(w)$, then $i$ appears exactly once in all elements of $R(w)$, and $i$ is always unconfined.
\end{lemma}

\begin{proof}
If $i$ is unconfined in $\rw{s}$, then $i$ cannot be part of any braid moves applied to $\rw{s}$. No sequence of commutation moves can produce another copy of $i$, so $i$ appears exactly once in all elements of $R(w)$. Because $\s_i$ does not commute with $\s_{i\pm 1}$, no sequence of commutation moves will land $i$ between two copies of $i+1$ or $i-1$. Thus $i$ remains unconfined in all elements of $R(w)$.
\end{proof}

Thus we can talk about a letter $i$ being ``unconfined in elements of $R(w)$.'' Unconfined letters are relevant to the architecture of the Bruhat order, particularly in the context of prisms.

\begin{proposition}\label{prop:no confines}
If $B(w) \cong B(\s_i) \times B(v)$ as defined in Lemma~\ref{lem:new class means one factor atom and the other in the poset}, then $i$ is unconfined in elements of $R(w)$.
\end{proposition}

\begin{proof}
We will prove the contrapositive statement. Suppose that $i$ is not unconfined in elements of $R(w)$. Then, without loss of generality, we have $\s_i \prec u \preceq w$, where $u = \s_{i+1}\s_i\s_{i+1} = \s_i\s_{i+1}\s_i$. The permutation $u$ is not a prism, and hence $B(w)$ does not decompose in the desired sense.
\end{proof}

From this, we see how the unconfined quality of a letter filters downward in the principal order ideal of a prism. The result below follows immediately from the definition of the Bruhat order, because it is impossible to introduce any confinement while deleting letters from a reduced word.

\begin{lemma}\label{lem:unconfined propagates downward}
Fix a permutation $w \in \symm_n$ in which some $i \in \support(w)$ is unconfined. Then, for each $u \in B(w)$, either $i \not\in\support(u)$ or $i$ is unconfined in $u$.
\end{lemma}

We can now give a complete characterization of prism permutations in terms of their reduced words.

\begin{theorem}\label{thm:reduced word characterization}
A permutation $w$ is a prism if and only if there exists $i \in \support(w)$ that is unconfined in elements of $R(w)$.
\end{theorem}

\begin{proof}
If we start with a prism $w$, then the result follows from Proposition~\ref{prop:no confines}.

For the other direction, suppose that there is such an $i \in \support(w)$. Following Corollary~\ref{cor:appearing once}, let $v$ be the permutation obtained by deleting $i$ from a reduced word for $w$. Define an operation $q$ on $B(w)$ as follows. For any $u \in B(w)$, if $i \not\in \support(u)$ then $q(u) := u$. Otherwise, $q(u)$ is the permutation whose reduced word is obtained by deleting $i$ from any reduced word for $u$. By Lemma~\ref{lem:unconfined propagates downward} and Corollary~\ref{cor:appearing once}, this is well defined. Now define $\phi : B(w) \rightarrow \big(B(\s_i) \times B(v)\big)$ as
$$\phi : u \mapsto \begin{cases}
(e, q(u)) = (e,u) & \text{ if } i \not\in \support(u), \text{ and}\\
(\s_i, q(u)) & \text{ if } i \in \support(u).
\end{cases}$$

The map $\phi$ is certainly surjective. To establish that it is injective, we consider two cases of $\phi(u) = \phi(u')$. If $q(u) = u$, then certainly $u = u'$. Now suppose that $q(u) \neq u$ (and hence $q(u') \neq u'$). The letter $i$ is unconfined in $w$, meaning that if $i-1$ (resp., $i+1$) is in the support of $q(u) = q(u')$, then $i$ appears on the same side of all copies of $i-1$ (resp., $i+1$) in $u$ as it does in $u'$. These facts in combination mean that $u = u'$. Therefore $\phi$ is injective. 

Because relations in the Bruhat order correspond to deleting letters from reduced words, the bijection $\phi$ is an isomorphism. Therefore, by Lemma~\ref{lem:new class means one factor atom and the other in the poset}, the permutation $w$ is a prism. 
\end{proof}

Note that when a non-identity element $w$ is purely boolean, every letter in its reduced words appears exactly once, and thus every letter is unconfined.

Given $\rw{s} \in R(w)$ in which some $i$ is unconfined, we can use commutation moves to produce an equivalent characterization of prisms as follows.

\begin{corollary}\label{cor:new class word forms}
A permutation $w$ is a prism if and only if it has a reduced word with one of the following formats:
\begin{enumerate}[label=\textup{(\alph*)}]
\item \label{word: big i little}
$\rw{(\text{letters greater than }i) \ i \ (\text{letters less than }i)}$,
\item \label{word: little i big}
$\rw{(\text{letters less than }i) \ i \ (\text{letters greater than }i)}$,
\item \label{word: i big little}
$\rw{i \ (\text{letters greater than }i) \ (\text{letters less than }i)}$, or
\item \label{word: little big i}
$\rw{(\text{letters less than }i) \ (\text{letters greater than }i) \ i}$,
\end{enumerate}
where the parenthetical phrases could be empty.
\end{corollary}

From these characterizations, we can inductively derive a statement about the degree to which a permutation is boolean. 

\begin{corollary}\label{cor:how boolean}
Fix a permutation $w$, and let $d$ be the number of distinct letters that are unconfined in elements of $R(w)$. Then 
\begin{equation}\label{eqn:boolean index}
B(w) \cong \left(\vcenter{\hbox{\begin{tikzpicture}[scale=.5]
\foreach \y in {0,1} {\fill (0,\y) circle (3pt);}
\draw (0,0) -- (0,1);
\end{tikzpicture}}}\right)^d \times B(w')
\end{equation}
for some permutation $w' \prec w$ that is not a prism. 
\end{corollary}

When $w$ is purely boolean, this $d$ is its length; when $w$ is not a prism, this $d$ is $0$. In all cases, it is a consequence of the definition that $2^d$ divides $|B(w)|$.

\section{Characterizing prisms via patterns}\label{sec:pattern characterization}

Previous work has shown a strong connection between reduced words and permutation patterns (see, for example, \cite{tenner patt-bru, tenner rep-patt, tenner rep-range, tenner rwm}). Thus the characterization of prisms given in Theorem~\ref{thm:reduced word characterization} prompts one to wonder whether these elements can also be characterized by their patterns. In fact they can, and that characterization is quite specific about how the patterns must be contained. That precision can be captured, almost entirely, by \emph{mesh patterns}, which were defined by Br\"and\'en and Claesson in \cite{branden claesson} and have since been studied in many places, including \cite{ulfarsson marked mesh}.

\begin{definition}\label{defn:mesh pattern}
Fix a permutation $p \in \symm_k$ and consider its graph $G(p)$ as living in the grid $[0,k+1]\times[0,k+1]$. Shade a subset $M$ (the \emph{mesh}) of the cells in that grid. A permutation $w$ contains the \emph{mesh pattern} $(p,M)$ if there is an occurrence of $G(p)$ in $G(w)$ in which no points of $G(w)$ appear in the regions that correspond to the mesh in $G(p)$.
\end{definition}

Mesh patterns are a generalization of classical patterns, in which case the mesh is empty.

\begin{example}\label{ex:mesh}
Let $w = 24153$ and consider two mesh patterns:
$$\mu = 
\vcenter{\hbox{\begin{tikzpicture}[scale=.4]
\fill[black!20] (0,1) rectangle (3,2);
\foreach \x in {1,2} {\draw (\x,0) -- (\x,3); \draw (0,\x) -- (3,\x);}
\foreach \x in {(1,2), (2,1)} {\filldraw \x circle (4pt);}
\end{tikzpicture}}}
\hspace{1in}
\mu' = 
\vcenter{\hbox{\begin{tikzpicture}[scale=.4]
\fill[black!20] (0,1) rectangle (4,2);
\fill[black!20] (2,2) rectangle (4,3); \fill[black!20] (2,2) rectangle (3,4);
\foreach \x in {1,2,3} {\draw (\x,0) -- (\x,4); \draw (0,\x) -- (4,\x);}
\foreach \x in {(1,2), (2,1), (3,3)} {\filldraw \x circle (4pt);}
\end{tikzpicture}}}\ .
$$
The permutation $w$ has four (circled) occurrences of the classical pattern $21$
$$
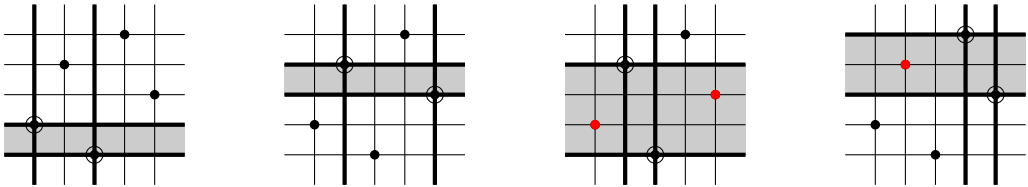
\begin{tikzpicture}[scale=.4]
\fill[black!20] (0,2) rectangle (6,1);
\foreach \x in {1,2} {\draw[ultra thick] (0,\x) -- (6,\x);}
\foreach \x in {1,3} {\draw[ultra thick] (\x,0) -- (\x,6);}
\foreach \x in {(1,2), (3,1)} {\draw \x circle (8pt);}
\foreach \x in {1,2,3,4,5} {\draw (\x,0) -- (\x,6); \draw (0,\x) -- (6,\x);}
\foreach \x in {(1,2), (2,4), (3,1), (4,5), (5,3)} {\filldraw \x circle (4pt);}
\end{tikzpicture}
\hspace{.5in}
\begin{tikzpicture}[scale=.4]
\fill[black!20] (0,4) rectangle (6,3);
\foreach \x in {3,4} {\draw[ultra thick] (0,\x) -- (6,\x);}
\foreach \x in {2,5} {\draw[ultra thick] (\x,0) -- (\x,6);}
\foreach \x in {(2,4), (5,3)} {\draw \x circle (8pt);}
\foreach \x in {1,2,3,4,5} {\draw (\x,0) -- (\x,6); \draw (0,\x) -- (6,\x);}
\foreach \x in {(1,2), (2,4), (3,1), (4,5), (5,3)} {\filldraw \x circle (4pt);}
\end{tikzpicture}
\hspace{.5in}
\begin{tikzpicture}[scale=.4]
\fill[black!20] (0,4) rectangle (6,1);
\foreach \x in {1,4} {\draw[ultra thick] (0,\x) -- (6,\x);}
\foreach \x in {2,3} {\draw[ultra thick] (\x,0) -- (\x,6);}
\foreach \x in {(2,4), (3,1)} {\draw \x circle (8pt);}
\foreach \x in {1,2,3,4,5} {\draw (\x,0) -- (\x,6); \draw (0,\x) -- (6,\x);}
\foreach \x in {(1,2), (2,4), (3,1), (4,5), (5,3)} {\filldraw \x circle (4pt);}
\filldraw[red] (1,2) circle (4pt);
\filldraw[red] (5,3) circle (4pt);
\end{tikzpicture}
\hspace{.5in}
\begin{tikzpicture}[scale=.4]
\fill[black!20] (0,5) rectangle (6,3);
\foreach \x in {3,5} {\draw[ultra thick] (0,\x) -- (6,\x);}
\foreach \x in {4,5} {\draw[ultra thick] (\x,0) -- (\x,6);}
\foreach \x in {(4,5), (5,3)} {\draw \x circle (8pt);}
\foreach \x in {1,2,3,4,5} {\draw (\x,0) -- (\x,6); \draw (0,\x) -- (6,\x);}
\foreach \x in {(1,2), (2,4), (3,1), (4,5), (5,3)} {\filldraw \x circle (4pt);}
\filldraw[red] (2,4) circle (4pt);
\end{tikzpicture}$$
but only the first and second of these are occurrences of the mesh pattern $\mu$. The permutation $w$ has three occurrences of the classical pattern $213$
$$
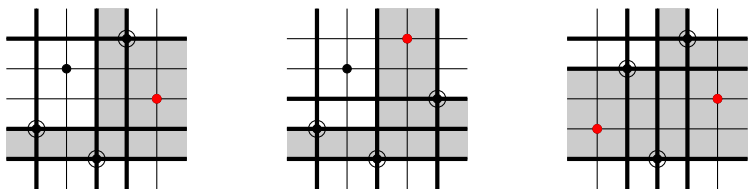
\begin{tikzpicture}[scale=.4]
\fill[black!20] (0,2) rectangle (6,1);
\fill[black!20] (3,1) rectangle (6,5); \fill[black!20] (3,1) rectangle (4,6);
\foreach \x in {1,2,5} {\draw[ultra thick] (0,\x) -- (6,\x);}
\foreach \x in {1,3,4} {\draw[ultra thick] (\x,0) -- (\x,6);}
\foreach \x in {(1,2), (3,1), (4,5)} {\draw \x circle (8pt);}
\foreach \x in {1,2,3,4,5} {\draw (\x,0) -- (\x,6); \draw (0,\x) -- (6,\x);}
\foreach \x in {(1,2), (2,4), (3,1), (4,5), (5,3)} {\filldraw \x circle (4pt);}
\filldraw[red] (5,3) circle (4pt);
\end{tikzpicture}
\hspace{.5in}
\begin{tikzpicture}[scale=.4]
\fill[black!20] (0,2) rectangle (6,1);
\fill[black!20] (3,1) rectangle (6,3); \fill[black!20] (3,1) rectangle (5,6);
\foreach \x in {1,2,3} {\draw[ultra thick] (0,\x) -- (6,\x);}
\foreach \x in {1,3,5} {\draw[ultra thick] (\x,0) -- (\x,6);}
\foreach \x in {(1,2), (3,1), (5,3)} {\draw \x circle (8pt);}
\foreach \x in {1,2,3,4,5} {\draw (\x,0) -- (\x,6); \draw (0,\x) -- (6,\x);}
\foreach \x in {(1,2), (2,4), (3,1), (4,5), (5,3)} {\filldraw \x circle (4pt);}
\filldraw[red] (4,5) circle (4pt);
\end{tikzpicture}
\hspace{.5in}
\begin{tikzpicture}[scale=.4]
\fill[black!20] (0,4) rectangle (6,1);
\fill[black!20] (3,1) rectangle (6,5); \fill[black!20] (3,1) rectangle (4,6);
\foreach \x in {1,4,5} {\draw[ultra thick] (0,\x) -- (6,\x);}
\foreach \x in {2,3,4} {\draw[ultra thick] (\x,0) -- (\x,6);}
\foreach \x in {(2,4), (3,1), (4,5)} {\draw \x circle (8pt);}
\foreach \x in {1,2,3,4,5} {\draw (\x,0) -- (\x,6); \draw (0,\x) -- (6,\x);}
\foreach \x in {(1,2), (2,4), (3,1), (4,5), (5,3)} {\filldraw \x circle (4pt);}
\filldraw[red] (1,2) circle (4pt);
\filldraw[red] (5,3) circle (4pt);
\end{tikzpicture}$$
and none of them are occurrences of $\mu'$. Therefore $w$ contains $\mu$ and avoids $\mu'$. In each figure above, any points of $G(w)$ that land in a mesh have been colored red. 
\end{example}

We now introduce what we call \emph{calibrated} mesh patterns, in which we are allowed to specify positions and values in an occurrence.

\begin{definition}\label{defn:calibrated}
A \emph{calibrated} mesh pattern $C$ is a mesh pattern $\mu$ in which positions and values may be labeled with positive integers. A permutation $w$ contains $C$ if $w$ has an occurrence of the mesh pattern $\mu$ in which any position in $\mu$ marked ``$x$'' appears in position $x$ in $w$, and any value in $\mu$ marked ``$y$'' is represented by the value $y$ in $w$.
\end{definition}

Calibration requirements on a mesh pattern could be achieved using a union of other mesh patterns (for example, avoidance of the calibrated pattern $C$ in Example~\ref{ex:calibrated} is equivalent to avoidance of twelve mesh patterns whose underlying classical patterns are in $\symm_4$), but calibration has an efficiency and, indeed, a clarity that such an alternative representation notably lacks. In \cite{ulfarsson}, \'Ulfarsson proposed the possible utility of attaching rules to the cells in a mesh pattern, like ``must contain at least three points of $G(w)$'' or ``must avoid $321$,'' and calibration fits that model.

\begin{example}\label{ex:calibrated}
Consider the calibrated mesh patterns:
$$\vcenter
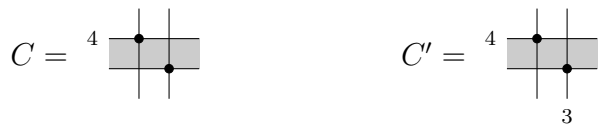
{\hbox{\begin{tikzpicture}[scale=.4]
\draw (-1,1.5) node[left] {$C=$};
\fill[black!20] (0,1) rectangle (3,2);
\foreach \x in {1,2} {\draw (\x,0) -- (\x,3); \draw (0,\x) -- (3,\x);}
\foreach \x in {(1,2), (2,1)} {\filldraw \x circle (4pt);}
\draw (0,2) node[left] {$\scriptstyle{4}$};
\draw[white] (2,0) node[below] {$\scriptstyle{3}$}; 
\end{tikzpicture}}}
\hspace{1in}
\vcenter{\hbox{\begin{tikzpicture}[scale=.4]
\draw (-1,1.5) node[left] {$C'=$};
\fill[black!20] (0,1) rectangle (3,2);
\foreach \x in {1,2} {\draw (\x,0) -- (\x,3); \draw (0,\x) -- (3,\x);}
\foreach \x in {(1,2), (2,1)} {\filldraw \x circle (4pt);}
\draw (0,2) node[left] {$\scriptstyle{4}$};
\draw (2,0) node[below] {$\scriptstyle{3}$};
\end{tikzpicture}}}$$
both of which are calibrations of the mesh pattern $\mu$ from Example~\ref{ex:mesh}. We saw in that example that the permutation $w = 42153$ contains $\mu$ in two ways, but only the second of those ways is an occurrence of the calibrated pattern $C$ (repeated in Figure~\ref{fig:calibrated example}), and neither of those $\mu$-occurrences is an occurrence of $C'$.
\begin{figure}[htbp]
\begin{tikzpicture}[scale=.4]
\fill[black!20] (0,4) rectangle (6,3);
\foreach \x in {3,4} {\draw[ultra thick] (0,\x) -- (6,\x);}
\foreach \x in {2,5} {\draw[ultra thick] (\x,0) -- (\x,6);}
\foreach \x in {(2,4), (5,3)} {\draw \x circle (8pt);}
\foreach \x in {1,2,3,4,5} {\draw (\x,0) -- (\x,6); \draw (0,\x) -- (6,\x);}
\foreach \x in {(1,2), (2,4), (3,1), (4,5), (5,3)} {\filldraw \x circle (4pt);}
\draw (0,4) node[left] {required by $C \ \rightarrow$};
\end{tikzpicture}
\caption{The permutation $42153$ contains the calibrated mesh pattern $C$ described in Example~\ref{ex:calibrated}.}\label{fig:calibrated example}
\end{figure}
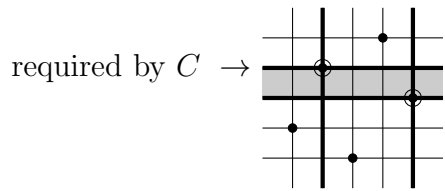
\end{example}

Our goal in this section is to translate Theorem~\ref{thm:reduced word characterization} into a statement about patterns in prisms. We start by considering a special class of these permutations.

\begin{proposition}\label{prop:new class with incomplete support}
Suppose that $w \in \symm_n$ is a prism with $i$ unconfined in elements of $R(w)$. If $|\{i\pm1\} \cap \support(w)| \le 1$, then $w$ has a $21$-pattern occurrence
\begin{itemize}
\item with values $\{i,i+1\}$, where $w(i+1) = i$ and/or $w(i) = i+1$, or
\item in positions $\{i,i+1\}$, where $w(i+1) = i$ and/or $w(i) = i+1$.
\end{itemize}
Put another way, if $i+1 \not\in \support(w)$ then $w$ contains 
$$\vcenter
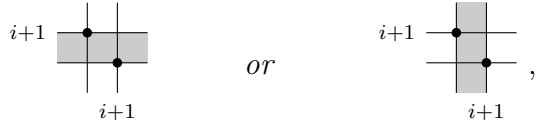
{\hbox{\begin{tikzpicture}[scale=.4]
\fill[black!20] (0,1) rectangle (3,2);
\foreach \x in {1,2} {\draw (\x,0) -- (\x,3); \draw (0,\x) -- (3,\x);}
\foreach \x in {(1,2), (2,1)} {\filldraw \x circle (4pt);}
\draw (0,2) node[left] {$\scriptstyle{i+1}$};
\draw (2,0) node[below] {$\scriptstyle{i+1}$}; 
\end{tikzpicture}}}
\hspace{.5in}
\text{or}
\hspace{.5in}
\vcenter{\hbox{\begin{tikzpicture}[scale=.4]
\fill[black!20] (1,0) rectangle (2,3);
\foreach \x in {1,2} {\draw (\x,0) -- (\x,3); \draw (0,\x) -- (3,\x);}
\foreach \x in {(1,2), (2,1)} {\filldraw \x circle (4pt);}
\draw (0,2) node[left] {$\scriptstyle{i+1}$};
\draw (2,0) node[below] {$\scriptstyle{i+1}$}; 
\end{tikzpicture}}}\ ,
$$
and if $i-1 \not\in \support(w)$ then $w$ contains
$$\vcenter
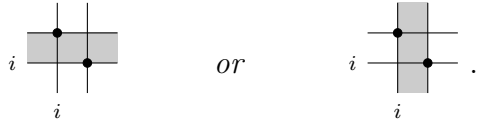
{\hbox{\begin{tikzpicture}[scale=.4]
\fill[black!20] (0,1) rectangle (3,2);
\foreach \x in {1,2} {\draw (\x,0) -- (\x,3); \draw (0,\x) -- (3,\x);}
\foreach \x in {(1,2), (2,1)} {\filldraw \x circle (4pt);}
\draw (0,1) node[left] {$\scriptstyle{i}$};
\draw (1,0) node[below] {$\scriptstyle{i}$}; 
\end{tikzpicture}}}
\hspace{.5in}
\text{or}
\hspace{.5in}
\vcenter{\hbox{\begin{tikzpicture}[scale=.4]
\fill[black!20] (1,0) rectangle (2,3);
\foreach \x in {1,2} {\draw (\x,0) -- (\x,3); \draw (0,\x) -- (3,\x);}
\foreach \x in {(1,2), (2,1)} {\filldraw \x circle (4pt);}
\draw (0,1) node[left] {$\scriptstyle{i}$};
\draw (1,0) node[below] {$\scriptstyle{i}$};
\end{tikzpicture}}}\ .
$$
\end{proposition}

\begin{proof}
Omitting at least one of $i\pm1$ from $\support(w)$ means that the cases of Corollary~\ref{cor:new class word forms} collapse to two cases: \ref{word: i big little} and~\ref{word: little big i}. The result follows from analyzing the kinds of permutations produced by these reduced words.
\end{proof}

In the calibrated mesh patterns listed in Proposition~\ref{prop:new class with incomplete support}, additional calibration labels are forced by the mesh. For example, the other horizontal line in the first figure must necessarily refer to the value $i$ in $w$.

Similar to Proposition~\ref{prop:new class with incomplete support}, when $\{i\pm 1 \} \subseteq \support(w)$, being a prism with unconfined $i$ implies a sort of displacement around the positions and values $\{i,i+1\}$.

\begin{proposition}\label{prop:new class with nearby support}
Suppose that $w \in \symm_n$ is a prism with $i$ unconfined in elements of $R(w)$. If $i \pm 1 \in \support(w)$, then $w$ contains one of the following calibrated mesh patterns.

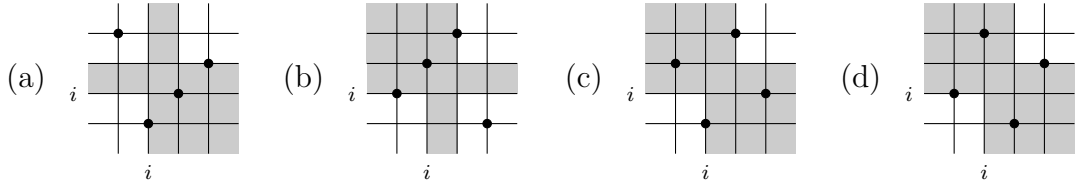
\begin{enumerate}[label=\textup{(\alph*)}]
\begin{minipage}{1.4in}
\item
$\vcenter{\hbox{\begin{tikzpicture}[scale=.4]
\fill[black!20] (0,2) rectangle (5,3);
\fill[black!20] (2,0) rectangle (3,5);
\fill[black!20] (3,3) rectangle (5,0);
\foreach \x in {1,2,3,4} {\draw (\x,0) -- (\x,5); \draw (0,\x) -- (5,\x);}
\foreach \x in {(1,4), (2,1), (3,2), (4,3)} {\filldraw \x circle (4pt);}
\draw (0,2) node[left] {$\scriptstyle{i}$};
\draw (2,0) node[below] {$\scriptstyle{i}$}; 
\draw[white] (2,5) node[above] {$\scriptstyle{i}$}; 
\end{tikzpicture}}}$
\end{minipage}
\begin{minipage}{1.4in}
\item
$\vcenter{\hbox{\begin{tikzpicture}[scale=.4]
\fill[black!20] (0,2) rectangle (5,3);
\fill[black!20] (2,0) rectangle (3,5);
\fill[black!20] (3,3) rectangle (0,5);
\foreach \x in {1,2,3,4} {\draw (\x,0) -- (\x,5); \draw (0,\x) -- (5,\x);}
\foreach \x in {(1,2), (2,3), (3,4), (4,1)} {\filldraw \x circle (4pt);}
\draw (0,2) node[left] {$\scriptstyle{i}$};
\draw (2,0) node[below] {$\scriptstyle{i}$}; 
\draw[white] (2,5) node[above] {$\scriptstyle{i}$}; 
\end{tikzpicture}}}$
\end{minipage}
\begin{minipage}{1.4in}
\item
$\vcenter{\hbox{\begin{tikzpicture}[scale=.4]
\fill[black!20] (0,2) rectangle (5,3);
\fill[black!20] (2,0) rectangle (3,5);
\fill[black!20] (3,3) rectangle (0,5);
\fill[black!20] (3,3) rectangle (5,0);
\foreach \x in {1,2,3,4} {\draw (\x,0) -- (\x,5); \draw (0,\x) -- (5,\x);}
\foreach \x in {(1,3), (2,1), (3,4), (4,2)} {\filldraw \x circle (4pt);}
\draw (0,2) node[left] {$\scriptstyle{i}$};
\draw (2,0) node[below] {$\scriptstyle{i}$}; 
\draw[white] (2,5) node[above] {$\scriptstyle{i}$}; 
\end{tikzpicture}}}$
\end{minipage}
\begin{minipage}{1.4in}
\item
$\vcenter{\hbox{\begin{tikzpicture}[scale=.4]
\fill[black!20] (0,2) rectangle (5,3);
\fill[black!20] (2,0) rectangle (3,5);
\fill[black!20] (3,3) rectangle (0,5);
\fill[black!20] (3,3) rectangle (5,0);
\foreach \x in {1,2,3,4} {\draw (\x,0) -- (\x,5); \draw (0,\x) -- (5,\x);}
\foreach \x in {(1,2), (2,4), (3,1), (4,3)} {\filldraw \x circle (4pt);}
\draw (0,2) node[left] {$\scriptstyle{i}$};
\draw (2,0) node[below] {$\scriptstyle{i}$}; 
\draw[white] (2,5) node[above] {$\scriptstyle{i}$}; 
\end{tikzpicture}}}$
\end{minipage}
\end{enumerate}
\end{proposition}

\begin{proof}
The labels of the patterns in the statement of this result corresponds to the labels of the word forms listed in Corollary~\ref{cor:new class word forms}. Consider case~\ref{word: big i little}, where $\rw{s} \in R(w)$ with
$$\rw{s} = \rw{(\text{letters greater than }i) \ i \ (\text{letters less than }i)}.$$
We can thus construct $w$ by reading $\rw{s}$ from left to right and acting on the positions of the identity permutation. In particular, because $i+1 \in \support(w)$, the reduced word $\rw{(\text{letters greater than }i)}$ produces
$$1 \ 2 \ \cdots \ i-1 \ i \ y \ \framebox{ \text{permutation of } $[i+1,n]\setminus \{y\}$ }$$
for some $y > i+1$. Next multiplying on the right by $\rw{i}$ produces 
$$1 \ 2 \ \cdots \ i-1 \ y \ i \ \framebox{ \text{permutation of } $[i+1,n]\setminus \{y\}$ }.$$
Finally, multiplying by $\rw{(\text{letters less than }i)}$ produces 
$$\framebox{ \text{permutation of } $[1,i-1] \cup \{y\} \setminus \{x\}$ } \ x \ i \ \framebox{ \text{permutation of } $[i+1,n]\setminus \{y\}$ },$$
where $x \le i-1$ due to the fact that $i-1 \in \support(w)$. The substring
$$y \ \ x \ \ i \ \ i+1$$
forms an occurrence of the calibrated mesh pattern (a) above. In particular, the requirement that $y$ is the only value larger than $i$ appearing to the left of $w(x) = i$ is forced by the calibrations and the mesh: values less than $i$ must all appear in the first $i$ positions of the permutation. The value $y$ also appears in those positions, so the other $i-1$ positions are forced to be exactly the values $[1,i-1]$.

The other cases follow from analogous arguments.
\end{proof}

It is interesting to note that the four patterns in the statement of Proposition~\ref{prop:new class with nearby support} have some common characteristics: namely, the calibrated position $i$ and value $i$, and the meshes' ``central cross,'' which force the additional calibration of position $i+1$ and value $i+1$.

We now use these results to give a pattern characterization for prisms, with the assistance of Lemma~\ref{lem:support rules}.

\begin{theorem}\label{thm:pattern characterization}
A permutation $w \in \symm_n$ is a prism if and only if there exists an $i$ for which $w$ contains one or more of the following calibrated mesh patterns.

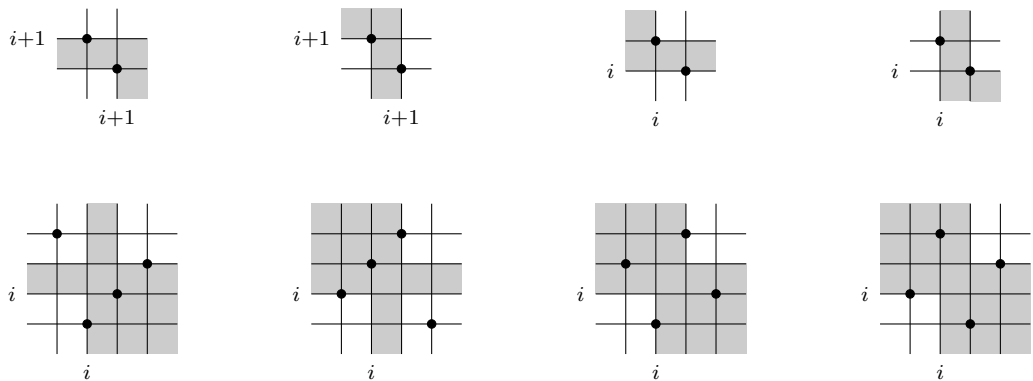
\begin{center}
{\renewcommand{\arraystretch}{6}\begin{tabular}{m{1.35in}m{1.35in}m{1.35in}m{1.35in}}
$\begin{tikzpicture}[scale=.4]
\fill[black!20] (0,1) rectangle (3,2);
\fill[black!20] (2,1) rectangle (3,0);
\foreach \x in {1,2} {\draw (\x,0) -- (\x,3); \draw (0,\x) -- (3,\x);}
\foreach \x in {(1,2), (2,1)} {\filldraw \x circle (4pt);}
\draw (0,2) node[left] {$\scriptstyle{i+1}$};
\draw (2,0) node[below] {$\scriptstyle{i+1}$}; 
\draw[white] (-1,0) node[left] {$\scriptstyle{i}$}; 
\end{tikzpicture}$
	&
$\begin{tikzpicture}[scale=.4] 
\fill[black!20] (1,0) rectangle (2,3);
\fill[black!20] (0,3) rectangle (1,2);
\foreach \x in {1,2} {\draw (\x,0) -- (\x,3); \draw (0,\x) -- (3,\x);}
\foreach \x in {(1,2), (2,1)} {\filldraw \x circle (4pt);}
\draw (0,2) node[left] {$\scriptstyle{i+1}$};
\draw (2,0) node[below] {$\scriptstyle{i+1}$}; 
\draw[white] (-1,0) node[left] {$\scriptstyle{i}$}; 
\end{tikzpicture}$
	&
$\begin{tikzpicture}[scale=.4] 
\fill[black!20] (0,1) rectangle (3,2);
\fill[black!20] (0,3) rectangle (1,2);
\foreach \x in {1,2} {\draw (\x,0) -- (\x,3); \draw (0,\x) -- (3,\x);}
\foreach \x in {(1,2), (2,1)} {\filldraw \x circle (4pt);}
\draw (0,1) node[left] {$\scriptstyle{i}$};
\draw (1,0) node[below] {$\scriptstyle{i}$}; 
\draw[white] (-1,0) node[left] {$\scriptstyle{i}$}; 
\end{tikzpicture}$
	&
$\begin{tikzpicture}[scale=.4]
\fill[black!20] (1,0) rectangle (2,3);
\fill[black!20] (2,1) rectangle (3,0);
\foreach \x in {1,2} {\draw (\x,0) -- (\x,3); \draw (0,\x) -- (3,\x);}
\foreach \x in {(1,2), (2,1)} {\filldraw \x circle (4pt);}
\draw (0,1) node[left] {$\scriptstyle{i}$};
\draw (1,0) node[below] {$\scriptstyle{i}$};
\draw[white] (-1,0) node[left] {$\scriptstyle{i}$}; 
\end{tikzpicture}$
\\

$\begin{tikzpicture}[scale=.4] 
\fill[black!20] (0,2) rectangle (5,3);
\fill[black!20] (2,0) rectangle (3,5);
\fill[black!20] (3,3) rectangle (5,0);
\foreach \x in {1,2,3,4} {\draw (\x,0) -- (\x,5); \draw (0,\x) -- (5,\x);}
\foreach \x in {(1,4), (2,1), (3,2), (4,3)} {\filldraw \x circle (4pt);}
\draw (0,2) node[left] {$\scriptstyle{i}$};
\draw (2,0) node[below] {$\scriptstyle{i}$}; 
\end{tikzpicture}$
	&
$\begin{tikzpicture}[scale=.4]
\fill[black!20] (0,2) rectangle (5,3);
\fill[black!20] (2,0) rectangle (3,5);
\fill[black!20] (3,3) rectangle (0,5);
\foreach \x in {1,2,3,4} {\draw (\x,0) -- (\x,5); \draw (0,\x) -- (5,\x);}
\foreach \x in {(1,2), (2,3), (3,4), (4,1)} {\filldraw \x circle (4pt);}
\draw (0,2) node[left] {$\scriptstyle{i}$};
\draw (2,0) node[below] {$\scriptstyle{i}$}; 
\end{tikzpicture}$
	&
$\begin{tikzpicture}[scale=.4]
\fill[black!20] (0,2) rectangle (5,3);
\fill[black!20] (2,0) rectangle (3,5);
\fill[black!20] (3,3) rectangle (0,5);
\fill[black!20] (3,3) rectangle (5,0);
\foreach \x in {1,2,3,4} {\draw (\x,0) -- (\x,5); \draw (0,\x) -- (5,\x);}
\foreach \x in {(1,3), (2,1), (3,4), (4,2)} {\filldraw \x circle (4pt);}
\draw (0,2) node[left] {$\scriptstyle{i}$};
\draw (2,0) node[below] {$\scriptstyle{i}$}; 
\end{tikzpicture}$
	&
$\begin{tikzpicture}[scale=.4] 
\fill[black!20] (0,2) rectangle (5,3);
\fill[black!20] (2,0) rectangle (3,5);
\fill[black!20] (3,3) rectangle (0,5);
\fill[black!20] (3,3) rectangle (5,0);
\foreach \x in {1,2,3,4} {\draw (\x,0) -- (\x,5); \draw (0,\x) -- (5,\x);}
\foreach \x in {(1,2), (2,4), (3,1), (4,3)} {\filldraw \x circle (4pt);}
\draw (0,2) node[left] {$\scriptstyle{i}$};
\draw (2,0) node[below] {$\scriptstyle{i}$}; 
\end{tikzpicture}$
\end{tabular}}
\end{center}
\end{theorem}

\begin{proof}
One direction of the result follows from Propositions~\ref{prop:new class with incomplete support} and~\ref{prop:new class with nearby support}, together with Lemma~\ref{lem:support rules}.

For the other direction, we will show that $w$ has a reduced word of one of the forms described in Corollary~\ref{cor:new class word forms}, and the result will follow. We will prove the result for the second, fifth, and eighth calibrated patterns listed above and leave the other, symmetric, arguments to the reader. 

First suppose that $w$ contains the calibrated mesh pattern
$$
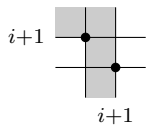
\begin{tikzpicture}[scale=.4]
\fill[black!20] (1,0) rectangle (2,3);
\fill[black!20] (0,3) rectangle (1,2);
\foreach \x in {1,2} {\draw (\x,0) -- (\x,3); \draw (0,\x) -- (3,\x);}
\foreach \x in {(1,2), (2,1)} {\filldraw \x circle (4pt);}
\draw (0,2) node[left] {$\scriptstyle{i+1}$};
\draw (2,0) node[below] {$\scriptstyle{i+1}$}; 
\end{tikzpicture}$$
for some  $i$. This means that $\{w(1),\ldots, w(i+1)\} = [1,i+1]$ and so $i+1\not\in\support(w)$ by Lemma~\ref{lem:support rules}. Set $v := w \s_i$, in which $v(i+1) = w(i) = i+1$ and $v(i) = w(i+1) < i+1$. Then
\begin{align*}
\{v(1),\ldots, v(i)\} &= \{w(1),\ldots, w(i+1)\} \setminus \{w(i)\} \\
&= \{w(1),\ldots, w(i+1)\} \setminus \{i+1\}\\
&= \{1,\ldots, i\},
\end{align*}
and so $i \not\in \support(v)$ by Lemma~\ref{lem:support rules}. Therefore we can find $\rw{t} \in R(v)$ of the form 
$$\rw{(\text{letters less than }i) \ (\text{letters greater than }i+1)},$$
and the concatenation $\rw{t}\rw{i} \in R(w)$ has the form described in Corollary~\ref{cor:new class word forms}\ref{word: little big i}. Thus $w$ is a prism. 

Now suppose that $w$ contains the calibrated mesh pattern
$$
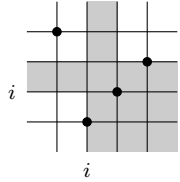
\begin{tikzpicture}[scale=.4]
\fill[black!20] (0,2) rectangle (5,3);
\fill[black!20] (2,0) rectangle (3,5);
\fill[black!20] (3,3) rectangle (5,0);
\foreach \x in {1,2,3,4} {\draw (\x,0) -- (\x,5); \draw (0,\x) -- (5,\x);}
\foreach \x in {(1,4), (2,1), (3,2), (4,3)} {\filldraw \x circle (4pt);}
\draw (0,2) node[left] {$\scriptstyle{i}$};
\draw (2,0) node[below] {$\scriptstyle{i}$}; 
\end{tikzpicture}$$
for some $i$. The mesh requirements mean that the values $[1,i-1]$ all appear in the first $i$ positions of $w$. Let $u \in \symm_n$ be such that $wu$ is obtained by putting $\{w(1),\ldots, w(i)\}$ into increasing order. In particular, $\support(u) \subseteq [1,i-1]$, and the permutation $wu$ fixes the values $[1,i-1]$ because of the mesh restrictions in the pattern. Moreover, in the permutation $v := wu\s_i$, the values $[1,i]$ are fixed. Thus $\support(v) \subseteq [i+1,n-1]$. Take any reduced words $\rw{s} \in R(u^{-1})$ and $\rw{t} \in R(v)$. Because $\support(u) = \support(u^{-1})$, the concatenation $\rw{t}\rw{i}\rw{s} \in R(w)$ has the format described in Corollary~\ref{cor:new class word forms}\ref{word: big i little} and hence $w$ is a prism. 

Finally, suppose that $w$ contains the calibrated mesh pattern 
$$
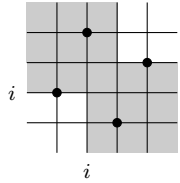
\begin{tikzpicture}[scale=.4]
\fill[black!20] (0,2) rectangle (5,3);
\fill[black!20] (2,0) rectangle (3,5);
\fill[black!20] (3,3) rectangle (0,5);
\fill[black!20] (3,3) rectangle (5,0);
\foreach \x in {1,2,3,4} {\draw (\x,0) -- (\x,5); \draw (0,\x) -- (5,\x);}
\foreach \x in {(1,2), (2,4), (3,1), (4,3)} {\filldraw \x circle (4pt);}
\draw (0,2) node[left] {$\scriptstyle{i}$};
\draw (2,0) node[below] {$\scriptstyle{i}$}; 
\end{tikzpicture}$$
for some $i$. Thus
$$\{w(1),\ldots, w(i-1)\} \cup \{w(i+1)\} = [1,i].$$
Consider $v := w \s_i$, in which $v(i+1) = w(i) > i+1$ and $v(i) = w(i+1) < i$. Then
$$\{v(1),\ldots, v(i)\} = \{w(1), \ldots, w(i+1)\} \setminus \{w(i)\} = [1,i],$$
and so Lemma~\ref{lem:support rules} means that $i \not\in \support(v)$. It follows that $v$ has a reduced word $\rw{t} \in R(v)$ of the form $\rw{(\text{letters less than }i) \ (\text{letters greater than }i)}$, and the concatenation $\rw{t}\rw{i} \in R(w)$ has the form described in Corollary~\ref{cor:new class word forms}\ref{word: little big i}. Thus $w$ is a prism. 

The remaining cases can be proved with analogous arguments.
\end{proof}

The pattern-analogue of Corollary~\ref{cor:how boolean} holds in this setting, as well: the number of values $i$ satisfying the statement of Theorem~\ref{thm:pattern characterization} determines the exponent $d$ in Equation~\eqref{eqn:boolean index}.

\section{Further research}\label{sec:further directions}

A natural direction of study after this work is to explore properties of boolean elements that have analogues for prisms. To put a slightly different spin on it, we can ask: how can the known properties of boolean elements be used to shed light on properties of prisms? The works cited earlier have described numerous features of boolean elements, and the structure of these objects suggests that they might influence properties of the substantial class of elements defined in this paper.

Our work here is focused on the symmetric group, but the basic objects that we study (reduced words, patterns, the Bruhat order) have analogues in Coxeter groups of other types. In particular, boolean elements -- those whose principal order ideals in the Bruhat order are isomorphic to boolean algebras -- have already been characterized by reduced words and pattern avoidance in other types \cite{gao hanni, tenner patt-bru}, and perhaps those groups' prisms can be characterized in those languages as well. One can also consider the poset defined by the weak order on Coxeter group elements, and boolean ideals and intervals have been studied in that setting, too \cite{elder harris kretschmann martinez mori, tenner intervals}. The central question of this paper, to characterize prism elements, would be interesting to study in any of those contexts as well.

In terms of extending the study of prisms in $\symm_n$ under the Bruhat order, one natural goal is to enumerate these elements. We have counted various classes related to prisms (including prisms, prisms that are not purely boolean, and non-prism elements) for $n \le 10$, and none of these sequences currently appear in \cite{oeis}. Permutations that are purely boolean have an attractive enumeration (they are the odd-indexed Fibonacci numbers \cite[A001519]{oeis}) and it is vexing to, as yet, have no ``nice'' enumeration of the prisms.

For another avenue of study related to this work, recall the permutation $w' \prec w$ discussed in Corollary~\ref{cor:how boolean}. It would be interesting to explore the relationship between $w'$ and $w$. Or, for another perspective, it could be fruitful to explore the ``$w'$-prisms'' for a fixed $w'$; namely, the collection
$$\left\{w : w \succ w' \text{ and there exists $d_w \ge 1$ such that } B(w) \cong \left(\vcenter{\hbox{\begin{tikzpicture}[scale=.5]
\foreach \y in {0,1} {\fill (0,\y) circle (3pt);}
\draw (0,0) -- (0,1);
\end{tikzpicture}}}\right)^{d_w} \times B(w')\right\}.$$
This brings to mind some of the earlier work looking at the collection of all boolean elements in a Coxeter group (for example, the topological properties of this subposet were studied in \cite{ragnarsson tenner homotopy, ragnarsson tenner homology}). Perhaps the analogous class for prisms would be the collection of all $w'$-prisms for a given $w'$. 

Finally, the introduction of calibrated patterns suggests their utility in other settings. To start with, existing results could benefit from this language. For example, \cite[Theorem~4.1]{tenner rep-range} is about the maximum number of times that the letter $k$ can appear in elements of $R(w)$, and it can be restated in terms of occurrences of the calibrated patterns
$$\left\{
\vcenter{\hbox{\begin{tikzpicture}[scale=.4] 
\foreach \x in {1,2,3} {\draw (\x,0) -- (\x,4); \draw (0,\x) -- (4,\x);}
\foreach \x in {(1,3), (2,2), (3,1)} {\filldraw \x circle (4pt);}
\draw (0,1) node[left] {$\scriptstyle{i}$};
\draw (0,3) node[left] {$\scriptstyle{j}$}; 
\end{tikzpicture}}}
\hspace{.1in}
\text{or}
\hspace{.1in}
\vcenter{\hbox{\begin{tikzpicture}[scale=.4] 
\foreach \x in {1,2,3,4} {\draw (\x,0) -- (\x,5); \draw (0,\x) -- (5,\x);}
\foreach \x in {(1,3), (2,4), (3,1), (4,2)} {\filldraw \x circle (4pt);}
\draw (0,1) node[left] {$\scriptstyle{i}$};
\draw (0,4) node[left] {$\scriptstyle{j}$}; 
\end{tikzpicture}}}\right\}
\hspace{.25in}
\text{and}
\hspace{.25in}
\left\{
\vcenter{\hbox{\begin{tikzpicture}[scale=.4] 
\foreach \x in {1,2,3} {\draw (\x,0) -- (\x,4); \draw (0,\x) -- (4,\x);}
\foreach \x in {(1,3), (2,2), (3,1)} {\filldraw \x circle (4pt);}
\draw (1,0) node[below] {$\scriptstyle{i}$};
\draw (3,0) node[below] {$\scriptstyle{j}$}; 
\end{tikzpicture}}}
\hspace{.1in}
\text{or}
\hspace{.1in}
\vcenter{\hbox{\begin{tikzpicture}[scale=.4] 
\foreach \x in {1,2,3,4} {\draw (\x,0) -- (\x,5); \draw (0,\x) -- (5,\x);}
\foreach \x in {(1,3), (2,4), (3,1), (4,2)} {\filldraw \x circle (4pt);}
\draw (2,0) node[below] {$\scriptstyle{i}$};
\draw (3,0) node[below] {$\scriptstyle{j}$}; 
\end{tikzpicture}}}\right\}
$$
for pairs $(i,j)$ with $i \le k < j$. Likewise, there might be new phenomena that are characterized by some sort of calibrated pattern containment or avoidance. It would also be interesting to understand enumerations related to calibrated patterns: either the number of permutations containing/avoiding a given calibrated pattern, or the number of times that a particular calibrated pattern occurs in a permutation, as studied in \cite{berman tenner}.

\section*{Acknowledgements}

I am grateful for the thoughtful and constructive suggestions of an anonymous referee, including to simplify the name of the objects studied in this work.

\end{document}